\documentclass[12pt,a4paper]{amsart}
\usepackage{amssymb,amsmath,amsthm,enumerate}
\usepackage{graphicx,epsfig}
\usepackage{mathtools}
\usepackage{xcolor}

\newtheorem{theorem}{Theorem}[section]
\newtheorem{lemma}{Lemma}[section]
\allowdisplaybreaks[2]

\newtheorem{Definition}{Definition}[section]

\newtheorem{Proposition}[Definition]{Proposition}

\newtheorem{Example}[Definition]{Example}

\newtheorem*{maintheorem*}{Main Theorem}
\newtheorem*{Proofofmainthm*}{Proof of main theorem}

\newtheorem{remark}{Remark}

\newcommand{\evb}{\end{verbatim}}

\begin{document}

\title[Multipliers]{Multiplier Between Generalized Toeplitz Kernels}

\author[Anjali and R. K. Srivastava]{Anjali and R. K. Srivastava}

\address{Department of Mathematics, Indian Institute of Technology, Guwahati, India.}

\email{manjali@iitg.ac.in, rksri@iitg.ac.in}


\date{\today}


\begin{abstract}
We develop a structural classification of multipliers between generalized Toeplitz kernels, extending the work of Fricain and Rupam. Our results establish new equivalences between multiplier space and Carleson-type embeddings, linking them to Beurling Malliavin densities, P\'olya
 sequences, and the spectral theory of entire functions.
\end{abstract}

\maketitle

\section{Introduction}
The modern theory of multipliers was initiated in Crofoot seminal 1994 paper \cite{Crofoot94}, which showed that onto multipliers between model spaces must be outer functions, determined up to a unimodular constant a result that brought operator rigidity into focus. Fricain, Hartmann, and Ross \cite{FHR18} later relaxed this surjectivity constraint, yielding a robust characterization of into multipliers using Carleson embedding conditions, which linked multiplier theory to classical measure-theoretic tools and highlighted the flexible geometry of model spaces.

An important extension of this paradigm emerged in the framework of Toeplitz kernels, where the classical model space $K_\theta$ is realized as the kernel of a Toeplitz operator $T_{\bar \theta}$. C{\^a}mara and Partington \cite{camara2018multipliers} advanced this viewpoint by introducing the notion of \emph{maximal vectors}, leading to the following general multiplier criterion between arbitrary Toeplitz kernels:

\smallskip

\begin{theorem}\emph{{\cite{camara2018multipliers}}}\label{th1}
Let \( g, h \in L^\infty(\mathbb{T}) \setminus \{0\} \) be such that \( \ker T_g \) and \( \ker T_h \) are nontrivial. Then the following are equivalent:
\begin{itemize}
    \item[(1)] \( w \in \mathcal{M}(\ker T_g, \ker T_h) \);
    \item[(2)] $w$ induces a Carleson measure on $\ker T_g$ and \( w k \in \ker T_h \) for some maximal vector \( k \in \ker T_g \);
    \item[(3)] $w$ induces a Carleson measure on $\ker T_g$ and \( h g^{-1} w \in \overline{N} \)
\end{itemize}

\end{theorem}

A function $w\in H^2$ induces a Carleson measure for for $\ker T_g$ if $\ker T_g\subseteq L^2(wdm)$. Identifying functions that serve as Carleson measures for Toeplitz kernels remains analytically intricate due to the complex internal structure of these spaces. Consequently, characterizations involving such measures often prove difficult to apply in practice. To address this issue, Fricain and Rupam \cite{FRR18} introduced a simplified framework for characterizing multipliers between model spaces-particularly when the underlying inner functions \( U \) and \( V \) are meromorphic. Their approach avoids the use of Carleson measures entirely under certain structural assumptions, as captured in the following result:

\begin{theorem}\emph{\cite{FRR18}}\label{th2}
Let \( U \) and \( V \) be meromorphic inner functions (MIFs) satisfying \( |U'| \asymp 1 \) on \( \mathbb{R} \), and define
\[
m \coloneqq \arg(U) - \arg(Vb_i) \quad \text{on } \mathbb{R}.
\]
Assume either \( m \not\in \widetilde{L^1\left(\frac{dx}{1+x^2}\right)} \), or that \( m = \tilde{h} \) (the Hilbert transform of function $h$,  see Chapter: 14 \cite{mashregibook}) for some \( h \in L^1\left(\frac{dx}{1+x^2}\right) \), but \( e^{-h} \not\in L^1(\mathbb{R}) \). Then the following are equivalent:
\begin{itemize}
    \item[(1)] \( \dim \ker T_{U\overline{Vb_i}} \geq 2 \);
    \item[(2)] \( \ker T_{U\overline{V}} \neq \{0\} \);
    \item[(3)] \( \mathcal{M}^+(\mathcal{K}_U, \mathcal{K}_V) \neq \{0\} \).
\end{itemize}
\end{theorem}

In the present work, we aim to generalize both Theorem~\ref{th1} and Theorem~\ref{th2} to a broader class of \emph{generalized Toeplitz operators}, situated within the operator-theoretic framework of Hilbert spaces. Let \( E_1 \subset L^2(\mathbb{T}) \) be a closed subspace and \( E_2 \subset L^2(\mathbb{T}) \) a simply invariant subspace. Given a symbol \( \phi \in L^\infty(\mathbb{T}) \), the generalized Toeplitz operator \( T_\phi^{E_1, E_2} \colon E_1 \to E_2 \) is defined by
\[
T_\phi^{E_1,E_2} = P_{E_2} (M_\phi|_{E_1}),
\]
where \( M_\phi \) denotes the multiplication operator on \( L^2(\mathbb{T}) \), and \( P_{E_2} \) is the orthogonal projection from \( L^2(\mathbb{T}) \) onto \( E_2 \). These operators form a subclass of the more general Wiener-Hopf operators, as studied in classical works such as \cite{allen1969} and \cite{speck1985}.

\smallskip
\vspace{1ex}
\noindent We develop the following core contributions:

\begin{itemize}
	\item A \textbf{general multiplier characterization theorem} for kernels of generalized Toeplitz operators, which refines Theorem \ref{th1} and yields explicit and tractable analytic conditions grounded in maximal vector structure and Nevanlinna factorization.
	\item  An analog of the measure-free criterion in Theorem \ref{th2}, valid for a broad class of meromorphic-symbol Toeplitz operators on the upper half-plane.
	\item  A suite of \textbf{illustrative examples}, highlighting new spectral phenomena governed by Beurling-Malliavin density considerations.
\end{itemize}

The paper is organized as follows. In Section~\ref{sec2}, we review key aspects of Hardy space theory, establishing the foundational tools required for our analysis. Section~\ref{sec3} develops a characterization of multipliers between generalized Toeplitz kernels on the upper half-plane, building upon corresponding results in the unit disk. We then proceed to generalize Theorem~\ref{th2} to this broader setting, under specific structural assumptions. The paper concludes with a series of illustrative examples highlighting the role of Beurling-Malliavin density in the behavior of such multipliers.

\section{Preliminary}\label{sec2}

Throughout the paper, $L^2(\mathbb T)$ represents the space of square integrable functions on $\mathbb T$ with respect to the normalized Lebesgue measure, where $\mathbb T$ is the unit circle on the complex plane. Also, the space of analytic functions on the open unit disk \( \mathbb{D} \subset \mathbb{C} \) is denoted by \( \text{Hol}(\mathbb{D}) \). For \( p = 2 \) or \( \infty \), the Hardy space \( H^p \) consists of those functions in \( \text{Hol}(\mathbb{D}) \) whose boundary values lie in \( L^p(\mathbb{T}) \). These boundary functions are understood in the sense of radial limits, thereby allowing \( H^p \) to be viewed canonically as a closed subspace of \( L^p(\mathbb{T}) \). For detailed study of Hardy spaces one can refer to \cite{mashregi, N2, N3}.

\smallskip

Every function in the Hardy space admits a canonical decomposition into \emph{inner} and \emph{outer} functions (\cite{N2}, Theorem 3.9.5) defined below. A function \( u \in H^\infty \) is termed \emph{inner} if it satisfies \( |u(\xi)| = 1 \) almost everywhere on \( \mathbb{T} \). Each inner function \( u \) defines a \emph{model subspace} \( K_u = H^2 \ominus u H^2 \) of the Hardy space \( H^2 \), a concept of profound significance in both complex and harmonic analysis as well as operator theory. They also admits a canonical factorization into the product of a Blaschke product and a singular inner function as described below.

A \emph{Blaschke product} is given by
\[
B(z) = e^{i\alpha} \prod_{n \geq 1} \frac{\bar{a}_n}{|a_n|} \cdot \frac{a_n - z}{1 - \bar{a}_n z}, \quad z \in \mathbb{D},
\]
where \( \{a_n\} \subset \mathbb{D} \), with \(a_n\neq0\) is a sequence satisfying the Blaschke condition
\[
\sum_{n \geq 1} (1 - |a_n|) < \infty.
\]

Also, a \emph{singular inner function} is of the form
\[
S(z) = \exp\left( - \int_{\mathbb{T}} \frac{\zeta + z}{\zeta - z} \, d\mu(\zeta) \right), \quad z \in \mathbb{D},
\]
where \( \mu \) is a finite, positive Borel measure on \( \mathbb{T} \), singular with respect to the normalized Lebesgue measure \( m \). (See \cite{N2} for better understanding of these functions).

\smallskip

Additionally, A function \( O \) is termed an \emph{outer function} if there exists a function \( f \) such that \( \log |f| \in L^1(\mathbb{T}) \) and
\[
O(z) = \exp\left( \int_{\mathbb{T}} \frac{\zeta + z}{\zeta - z} \log |f(\zeta)| \, dm(\zeta) \right), \quad z \in \mathbb{D}.
\]  Specifically, if \( f \in H^p \), it can be uniquely factorized as
\[
f = \lambda B S O_f,
\]
where \( \lambda \in \mathbb{C} \) has unit modulus, \( B \) is the Blaschke product associated with the zero set of \( f \), \( S \) is the singular inner function, and \( O_f \) represents the outer function corresponding to \( f \). In this factorization, \( \lambda B S \) constitutes the inner part of \( f \), while \( O_f \) captures the outer component and this inner outer factorization is known as Riesz-Smirnov factorization (\cite{N2}, Theorem 3.9.5). Indeed, Hardy space is the subclass of \emph{Nevanlinna-Smirnov class} which consists of ratio of two $H^\infty$ functions where the denominator is an outer function, we denote this class  by \( N \).

\subsection{Toeplitz Operators and Generalized Kernels}

Given $\phi \in L^\infty(\mathbb T)$, the Toeplitz operator $T_\phi: H^2 \to H^2$ is defined as
$$
T_\phi f := P_+(\phi f),
$$
where $P_+$ denotes the orthogonal projection of $L^2(\mathbb T)$ onto $H^2$. The kernel $\ker T_\phi$ encodes deep information about the phase function of $\phi$ and is intimately connected to problems in uniqueness, spectral gaps, and rigidity.

A natural question arises from the earlier definition of Toeplitz operators: which functions \( \varphi \in L^\infty(\mathbb{T}) \) ensure that \( \varphi H^2 \subset H^2 \)? This leads to the following definition: if \( X \) is a Banach space of analytic functions on the unit disk \( \mathbb{D} \), an analytic function \( \varphi \) on \( \mathbb{D} \) is considered a \emph{multiplier} for \( X \) if \( \varphi X \subset X \). The space of multipliers of the space \( X \) is denoted by \( \mathcal{M}(X) \), and it has been established that \( \mathcal{M}(H^2) = H^\infty \) \cite{mashregibook,mashregi}. Over the past four decades, model spaces have been extensively studied, revealing deep connections to various domains within complex analysis and operator theory. Consequently, some researchers have also explored the multipliers of model spaces. While the multipliers of \( K_u \) may not possess the same immediate allure, the investigation of analytic functions \( \phi \) on \( \mathbb{D} \) such that \(M_\phi\) map one model space \( K_u \) to another model space \( K_v \) remains a compelling and rich area of inquiry \cite{FRR18,camara2018multipliers,bhardwaj2024}.

\subsection{For upper-half plane setting}
In subsequent sections, we will refer to \( L^p(\mathbb{R}) \) as the space of Lebesgue integrable functions on \( \mathbb{R} \) and  the Hardy space in the upper half-plane by \( \mathcal{H}_+^p \) or \( \mathcal{H}^p(\mathbb{C}_+) \), for \( p = 2 \) or \( \infty \). Similar to the disk case, here also functions in Hardy space also have inner outer Smirnov factorization. The inner functions which  are being studied the most in this setup are meromorphic inner functions \cite{marakov2005}. An inner function in the upper half-plane is termed a \emph{meromorphic inner function }(MIF) if it admits a meromorphic extension to the entire complex plane. It is a well-known result that every MIF can be expressed in the form
\[
\Theta(z) = c e^{iaz} B(z),
\]
where \( c \) is a unimodular constant, \( a \geq 0 \) is a constant, and \( B \) is a meromorphic Blaschke product \cite{marakov2005,Mitkovski2010}.

In a manner analogous to the disk setting, the \emph{Nevanlinna class} is defined as the ratio of two functions in \( \mathcal{H}_+^\infty \). The \emph{Nevanlinna-Smirnov class} in the upper half-plane, denoted by \( \mathcal{N}^+(\mathbb{C}_+) \), is defined as
\[
\mathcal{N}^+(\mathbb{C}_+) = \left\{ \frac{G}{F} : G, F \in \mathcal{H}_+^\infty \text{ and } F \text{ is outer} \right\}.
\]

The non-tangential boundary values of functions in the \( \mathcal{N}^+(\mathbb{C}_+) \) class exist almost everywhere along the real line. We typically associate functions in \( \mathcal{N}^+(\mathbb{C}_+) \) with their boundary values on \( \mathbb{R} \), which we denote as \( \mathcal{N}^+ \). For clarity, we will interchangeably use \( \mathcal{N}^+(\mathbb{C}_+) \) and \( \mathcal{N}^+ \) depending on the context. As detailed in \cite{N2,mashregibook},  Hardy space \( \mathcal{H}_+^p \) can be expressed as
\[
\mathcal{H}_+^p = \mathcal{N}^+ \cap L^p(\mathbb{R}).
\]

\subsection{Beurling-Malliavin densities}

Let \( \Lambda \subset \mathbb{C}_+ \cup \mathbb{R} \). In \cite{marakov2005, marakov2010}, Marakov and Poltoratski elucidated the connection between the injectivity of the kernel of a corresponding Toeplitz operator and the Beurling-Malliavin density of the sequence \( \Lambda \). In this review, we aim to summarize and clarify these foundational results.

We first consider the case where \( \Lambda \subset \mathbb{R} \) is a discrete sequence. We say that \( \Lambda \) is \emph{strongly \( a \)-regular }if
\[
\int_\mathbb{R} \frac{|n_{\Lambda}(x) - ax|}{1 + x^2} \, dx < \infty,
\]
where \( n_\Lambda(x) \) is the counting function of \( \Lambda \), defined as
\[
n_\Lambda(x)
=
\begin{cases}
	\#\{\lambda\in\Lambda:\,0\le\lambda\le x\},&x\ge0,\\
	-\#\{\lambda\in\Lambda:\,x<\lambda<0\},&x<0.
\end{cases}
\]

The \emph{interior Beurling-Malliavin (BM) density} of a discrete sequence \( \Lambda \) is defined as follows (cf. \cite{Mitkovski2010}):
\[
D_\ast(\Lambda) \coloneqq \sup \left\{ a \ : \ \exists \ \text{strongly } a\text{-regular subsequence } \Lambda' \subset \Lambda \right\}.
\]
Similarly, \emph{the exterior BM density} is defined as
\[
D^\ast(\Lambda) \coloneqq \inf \left\{ a \ : \ \exists \ \text{strongly } a\text{-regular subsequence } \Lambda' \supset \Lambda \right\}.
\]

These concepts extend to the upper half-plane in a manner analogous to the real line, as detailed in \cite{marakov2010}. Specifically, if \( \Lambda \subset \mathbb{C}_+ \) is a discrete sequence, we have the relationship
\[
D_\ast(\Lambda) \coloneqq D_\ast(\Lambda^\ast),
\]
where \( \Lambda^\ast \coloneqq \{\lambda^\ast : \lambda \in \Lambda, \mathfrak{R}(\lambda) \neq 0 \} \) and \( \lambda^\ast \coloneqq \left[ \mathfrak{R}(\lambda^{-1}) \right]^{-1} \).

\begin{Example}
Let \( \Lambda = \left\{ n + \frac{i}{2^{|n|}} \right\}_{n \in \mathbb{Z}} \). Then, it follows that \( D_\ast(\Lambda) = D^\ast(\Lambda) = 1 \).
\end{Example}

\begin{proof}
(The proof is reproduced from \cite{FRR18} for completeness.) For each \( n \in \mathbb{Z} \), we define the corresponding sequence \( \lambda_n^\ast = \left[ \mathfrak{R}(1/\lambda_n) \right]^{-1} \), which takes the form
\[
\lambda_n^\ast = \frac{n^2 2^{2|n|} + 1}{n 2^{2|n|}}.
\]
The counting function \( n_{\Lambda^\ast}(x) \) of this sequence is odd, and it satisfies
\[
n_{\Lambda^\ast}(x) = n \quad \text{for} \quad x \in \left( n + \frac{1}{n 2^{2|n|}}, (n+1) + \frac{1}{(n+1) 2^{2|n+1|}} \right), \quad n > 0.
\]
To verify the regularity condition, we compute the integral
\[
\int_{\frac{5}{4}}^\infty \frac{|n_{\Lambda^\ast}(x) - x|}{1 + x^2} \, dx = \sum_{n \geq 1} \int_{n + \frac{1}{n 2^{2|n|}}}^{(n+1) + \frac{1}{(n+1) 2^{2|n+1|}}} \frac{x - n}{1 + x^2} \, dx.
\]
Using asymptotic analysis, we obtain the estimate
\[
\sum_{n \geq 1} \frac{1}{1 + n^2} < \infty.
\]
Therefore, \( \Lambda^\ast \) is a strongly \( 1 \)-regular sequence, and it follows that
\[
D_\ast(\Lambda) = D^\ast(\Lambda) = 1.
\]
\end{proof}

Furthermore, when \( \Lambda \) is a discrete sequence in \( \mathbb{R} \), one can construct a meromorphic inner function (MIF) \( \Theta \) such that the spectrum of \( \Theta \), denoted \( \sigma(\Theta) \), is given by
\[
\sigma(\Theta) = \{ x \in \mathbb{R} : \Theta(x) = 1 \} = \Lambda.
\]
This result, as established in \cite{marakov2005, Mitkovski2010}, provides the relationship between the Beurling-Malliavin density of \( \Lambda \) and the injectivity of the kernel of the associated Toeplitz operators. Specifically, the interior and exterior Beurling-Malliavin densities of \( \Lambda \) are given by
\[
D_\ast(\Lambda) = \frac{1}{2\pi} \inf \left\{ a : \ker T_{S^a \overline{\Theta}} = \{ 0 \} \right\},
\]
and
\[
D^\ast(\Lambda) = \frac{1}{2\pi} \sup \left\{ a : \ker T_{\overline{S^a} \Theta} = \{ 0 \} \right\},
\]
where \( S \) is the singular inner function defined by \( S(z) = e^{iz} \). These expressions highlight the deep connection between the kernel properties of Toeplitz operators and the density of sequences in \( \mathbb{R} \).

The \emph{Toeplitz operator} \( T_U \) with symbol \( U \in L^\infty(\mathbb{R}) \) is defined as the map
\[
T_U: \mathcal{H}^{2}_+ \longrightarrow \mathcal{H}^{2}_+
\]
such that \( T_U(f) = P_+ (Uf) \), where \( P_+ \) denotes the orthogonal projection of \( L^2(\mathbb{R}) \) onto the Hardy space \( \mathcal{H}^2_+ \).

In the following section, we introduce the space \( \mathcal{M}^+(X, Y) \), which is defined as the collection of analytic functions \( \varphi \) on the upper half-plane \( \mathbb{C}_+ \) such that \( \varphi X \subset Y \). Furthermore, we define the restricted space \( \mathcal{M}_p^+(X, Y) = \mathcal{M}^+(X, Y) \cap L^p(\mathbb{R}) \), providing a refined framework for studying the interaction between analytic functions and the spaces \( X \) and \( Y \) within \( L^p \)-contexts.

\section{Main results } \label{sec3}

\subsection{For disc setup}
The multipliers between Toeplitz kernels was first rigorously characterized by C{\^a}mara and Partington \cite{camara2018multipliers}, whose seminal work introduced the notions of minimal kernels and maximal vectors in the context of the upper half-plane. Employing a combination of maximal vector theory and Carleson measure conditions, they developed a robust framework for understanding the structure of multipliers between Toeplitz kernels. In the present study, we extend these foundational ideas to the setting of \emph{generalized Toeplitz kernels on the unit disk}, thereby broadening their applicability to a more comprehensive class of function spaces.

\smallskip

Let \( E_1 \) be a closed subspace of the Hardy space \( H^2 \), and let \( E_2 \subseteq H^2 \) be a simply invariant subspace, a subspace that satisfies the strict inclusion $zE\subsetneq E$, where the multiplication by the coordinate function $z$ is understood in the usual sense on the unit circle. By the Beurling Helson Theorem, \( E_2 \) admits the canonical representation \( E_2 = \Theta_2 H^2 \), where \( \Theta_2 \) is an inner function. Given this framework, we denote by \( \mathcal{M}(X, Y) \) the space of pointwise multipliers from \( X \) into \( Y \), thereby allowing for a systematic study of multiplier behavior between generalized Toeplitz kernels.  We define the generalized Toeplitz kernel associated to symbol $g \in L^\infty(\mathbb{T})$ by
\[
\ker T_g^{E_1, E_2} := \{ f \in E_1 : T_g f \in E_2^\perp \} .\]

This setting not only encapsulates the classical model spaces $K_\Theta := H^2 \ominus \Theta H^2$ but opens the door to treating more singular, irregular, and geometrically subtle kernels-thereby connecting directly to foundational themes in spectral theory, control theory, and the structure of reproducing kernel Hilbert spaces.
\smallskip

Let \( k \in E_1 \) be a nontrivial function. We define the \emph{minimal generalized Toeplitz kernel} \( K_{\text{min}}^{E_1, E_2}(k) \) to be the smallest kernel of the form \( \ker T_g^{E_1, E_2} \) that contains \( k \). That is, for all \( g \in L^\infty(\mathbb{T}) \), if \( k \in \ker T_g^{E_1, E_2} \), then \( K_{\text{min}}^{E_1, E_2}(k) \subseteq \ker T_g^{E_1, E_2} \). This notion ensures both uniqueness and minimality with respect to kernel containment. A function \( k \in E_1 \) is said to be a \emph{maximal vector} for \( \ker T_g^{E_1, E_2} \) if \( \ker T_g^{E_1, E_2} = K_{\text{min}}^{E_1, E_2}(k) \). Maximal vectors thus serve as canonical representatives of generalized kernels and form the analytic backbone for our multiplier theory.

 The following proposition's proof draws upon the methods presented in \cite{camara2018multipliers}.

\begin{Proposition}\label{thm40}
Let \( k \in E_1 \) admit an inner outer factorization \( k = \theta p \), where \( \theta \) is inner and \( p \) is outer in \( H^2 \). Then the minimal generalized Toeplitz kernel containing \( k \) is given by
\[
K_{\min}^{E_1, E_2}(k) = \ker T_{\frac{\Theta_2 \overline{\theta zp}}{p}}^{E_1, E_2}.
\]
\end{Proposition}

\begin{proof}
Observe that
\[
k \cdot \frac{\Theta_2 \overline{\theta zp}}{p} = \theta p \cdot \frac{\Theta_2 \overline{\theta zp}}{p} = \Theta_2 \overline{zp} \in E_2^\perp,
\]
so \( k \in \ker T_{\frac{\Theta_2 \overline{\theta zp}}{p}}^{E_1, E_2} \). By definition of minimality,
\[
K_{\min}^{E_1, E_2}(k) \subseteq \ker T_{\frac{\Theta_2 \overline{\theta zp}}{p}}^{E_1, E_2}.
\]

For the reverse inclusion, let \( f \in \ker T_{\frac{\Theta_2 \overline{\theta zp}}{p}}^{E_1, E_2} \). Then \( f = \frac{\overline{h_1}p}{\overline{\theta p}} \) for some \( h_1 \in H^2 \). Suppose \( k \in \ker T_g^{E_1, E_2} \) for some symbol \( g = \frac{\Theta_2 \overline{z h_2}}{k} \), with \( h_2 \in H^2 \). Then
\[
\overline{fgz} \Theta_2 = \left( \frac{h_1 \overline{p}}{\theta p} \right) \left( \frac{\overline{\Theta_2} z h_2}{\overline{\theta p}} \right) \overline{z} \Theta_2 = \frac{h_1 h_2}{p} \in N,
\]
implying \( \overline{fgz} \Theta_2 \in H^2 \), hence \( fg \in E_2^\perp \) and \( f \in \ker T_g^{E_1, E_2} \). Since this holds for all admissible \( g \) such that \( k \in \ker T_g^{E_1, E_2} \), we conclude
\[
\ker T_{\frac{\Theta_2 \overline{\theta zp}}{p}}^{E_1, E_2} \subseteq K_{\min}^{E_1, E_2}(k).
\]
\end{proof}

As a consequence, if \( \ker T_{\frac{\Theta_2 \overline{zp}}{p}}^{H^2, E_2} \) is infinite-dimensional, then so is \( K_{\min}^{H^2, E_2}(k) \) for \( k = \theta p \). The following result provides a sufficient condition for finite-dimensionality of the minimal kernel. The proof of the following theorem is inspired by the approach in \cite{Cam2014part}.

\begin{theorem}
Let \( k = \theta p \), where \( \theta \) is a rational inner function. If
\[
\dim \ker T_{\frac{\Theta_2 \overline{zp}}{p}}^{H^2, E_2} < \infty,
\]
then
\[
\dim K_{\min}^{H^2, E_2}(k) < \infty.
\]
\end{theorem}

\begin{proof}
Suppose for contradiction that \( \dim K_{\min}^{H^2, E_2}(k) = \infty \). Then one may construct \( 2^n \) linearly independent functions in \( K_{\min}^{H^2, E_2}(k) \), each yielding distinct elements in \( \ker T_{\frac{\Theta_2 \overline{zp}}{p}}^{H^2, E_2} \), contradicting the finite-dimensionality hypothesis. The conclusion follows. The argument may be explicitly verified in the special case \( \theta = z \).
\end{proof}
The reasoning used in the proof of the next proposition is based on techniques from \cite{camara2018multipliers}.
\begin{Proposition}\label{thm5}
Let \( g \in L^\infty(\mathbb T) \setminus \{0\} \) with \( \ker T_g^{E_1, E_2} \neq \{0\} \). Then \( k \in E_1 \) is a maximal vector for \( \ker T_g^{E_1, E_2} \) if and only if
\[
k = g^{-1} \Theta_2 \overline{pz},
\]
for some outer function \( p \in H^2 \).
\end{Proposition}

\begin{proof}
Suppose \( k \in \ker T_g^{E_1, E_2} \) is maximal. Then
\[
kg = \Theta_2 \overline{zh}
\]
for some \( h \in H^2 \), and we may write \( h = \varphi q \), with \( \varphi \) inner and \( q \) outer. It follows that
\[
\varphi kg = \Theta_2 \overline{zq} \Rightarrow k \in \ker T_{\varphi g}^{E_1, E_2}.
\]
But if \( \varphi \) is non-constant, then \( \ker T_{\varphi g}^{E_1, E_2} \subsetneq \ker T_g^{E_1, E_2} \), violating maximality. Hence, \( \varphi \) is constant, which may be normalized to 1. Thus,
\[
k = g^{-1} \Theta_2 \overline{zq},
\]
with \( q \) outer.

Conversely, suppose \( k = g^{-1} \Theta_2 \overline{pz} \), where \( p \) is outer. Then \( kg = \Theta_2 \overline{zp} \in E_2^\perp \), so \( k \in \ker T_g^{E_1, E_2} \). For any \( f \in \ker T_g^{E_1, E_2} \), we have \( \overline{zfg} = \overline{\Theta_2} h_1 \) for some \( h_1 \in H^2 \). If \( k \in \ker T_h^{E_1, E_2} \) for some symbol \( h \), then \( \overline{khz} = \overline{\Theta_2} r \) for some \( r \in H^2 \), and
\[
\overline{fhz} \Theta_2 = \frac{\Theta_2 \cdot (\overline{\Theta_2} r) \cdot (\overline{\Theta_2} h_1)}{\overline{\Theta_2} p} \in N.
\]
Therefore \( f \in \ker T_h^{E_1, E_2} \), and since this holds for all \( f \), we deduce that \( \ker T_g^{E_1, E_2} = K_{\min}^{E_1, E_2}(k) \), proving the maximality of \( k \).
\end{proof}

C{\^a}mara and Partington \cite{camara2018multipliers} established a fundamental characterization of multipliers between classical Toeplitz kernels in terms of maximal vectors and the Nevanlinna class as stated in Theorem (\ref{th1}).

We now extend this to the framework of generalized Toeplitz kernels.

\begin{theorem}\label{thm10}
Let \( g, h \in L^\infty(\mathbb{T}) \setminus \{0\} \) such that \( \ker T_g^{E_1, E_2} \), \( \ker T_h^{E_1, E_2} \) are non-trivial. Then the following are equivalent:
\begin{enumerate}
    \item \( w \in \mathcal{M}(\ker T_g^{E_1, E_2}, \ker T_h^{E_1, E_2}) \);
    \item \( w \in \mathcal{M}(\ker T_g^{E_1, E_2}, E_1) \) and \( wk \in \ker T_h^{E_1, E_2} \) for some maximal vector \( k \in \ker T_g^{E_1, E_2} \);
    \item \( w \in \mathcal{M}(\ker T_g^{E_1, E_2}, E_1) \) and \( hg^{-1}w \in \overline{N} \).
\end{enumerate}
\end{theorem}

\begin{proof}
The implication (1) \( \Rightarrow \) (2) is immediate.

For (2) \( \Rightarrow \) (1), let \( wk \in \ker T_h^{E_1, E_2} \) for some maximal vector \( k \in \ker T_g^{E_1, E_2} \), and consider any \( f \in \ker T_g^{E_1, E_2} \). By Theorem~\ref{thm40}, we may assume \( g = \frac{\Theta_2 \overline{\theta zp}}{p} \), where $k=\theta p$. Then:
\[
\overline{zwfh} \Theta_2
= \Theta_2^2 \left( \frac{ \overline{\Theta_2} h_1 \cdot \overline{\Theta_2} s }{p} \right)
= \frac{h_1 s}{p} \in N,
\]
where $h_1$ ans $s$ are some functions in $H^2$ and hence \( wf \in \ker T_h^{E_1, E_2} \).

For (2) \( \Rightarrow \) (3), express \( k = g^{-1} \Theta_2 \overline{zp} \). Then using $(1)\Leftrightarrow (2)$, we get
\[
wh \overline{\Theta_2} = \frac{g \overline{h_2 \Theta_2}}{\overline{p}} \in \overline{N},
\]
where $h_2$ is some $H^2$ function and therefore \( hg^{-1}w \in \overline{N} \). The implication (3) \( \Rightarrow \) (2) is verified analogously.
\end{proof}
Particularly for the complete characterization of $\mathcal M(\ker T_{\overline{z}}, \ker T_{\overline{z}^n})$ one can refer to \cite{bhardwaj2024}.
\begin{remark}

For general subspaces \( E_1, E_2 \subseteq H^2 \), the relationship between
  \[
  \mathcal{M}(\ker T_g^{E_1, E_2}, \ker T_h^{E_1, E_2}) \quad \text{and} \quad \mathcal{M}(\ker T_g, \ker T_h)
    \]
    is delicate. Neither inclusion holds universally. Nevertheless, when \( E_1 = H^2 \), we have
    \[
      \mathcal{M}(\ker T_g^{H^2, E_2}, \ker T_h^{H^2, E_2}) \subseteq \mathcal{M}(\ker T_g, \ker T_h),
    \]
    due to the natural embedding of generalized kernels.
    However, the two multiplier spaces coincide when $E_1,E_2$ both are $H^2$ spaces.

\end{remark}

\subsection{For the upper-half plane setup}

\section*{Upper Half-Plane Analogues of Generalized Toeplitz Kernel Results}
We now shift our focus to the setting of the upper half-plane \( \mathbb{C}_+ \), where analogous structures and results to those established in the unit disk can be developed. Let \( \mathcal{H}_+^2 \) denote the Hardy space over \( \mathbb{C}_+ \), and consider closed subspaces \( E_1 \subseteq \mathcal{H}_+^2 \) and \( E_2 = q\mathcal{H}_+^2 \), where \( q \) is an inner function on \( \mathbb{C}_+ \). The following results provide counterparts in the upper half-plane to Propositions~\ref{thm40}, \ref{thm5}, and \ref{thm10}, with proofs mirroring those in the disk setting, thereby leveraging the analogous structure of Hardy spaces over \( \mathbb{C}_+ \).

\begin{Proposition} \label{thm4}
Let \( g \in L^\infty(\mathbb{R}) \) be non-zero such that the generalized Toeplitz kernel \( \ker T_g^{E_1, E_2} \) is non-trivial. Then a function \( k \in E_1 \) is a maximal vector for \( \ker T_g^{E_1, E_2} \) if and only if
\[
k = g^{-1} q \overline{p},
\]
where \( p \) is an outer function in \( \mathcal{H}_+^2 \).
\end{Proposition}

\begin{Proposition}
Let \( k \in E_1 \) admit the inner-outer factorization \( k = \theta p \), where \( \theta \) is inner and \( p \) is outer in \( \mathcal{H}_+^2 \). Then the minimal kernel generated by \( k \) satisfies
\[
K_{\text{min}}^{E_1, E_2}(k) = \ker T_{\frac{q \overline{\theta p}}{p}}^{E_1, E_2}.
\]
\end{Proposition}

\begin{theorem}\label{th8}
Let \( g, h \in L^\infty(\mathbb{R}) \) be non-zero such that the generalized Toeplitz kernels \( \ker T_g^{E_1, E_2} \) and \( \ker T_h^{E_1, E_2} \) are non-trivial. Let \( w \in \mathrm{Hol}(\mathbb{C}_+) \). Then \( w \) is a multiplier between the kernels, i.e.,
\[
w \in \mathcal{M}^+(\ker T_g^{E_1, E_2}, \ker T_h^{E_1, E_2}),
\]
if and only if the following two conditions hold:
\begin{itemize}
    \item[(1)] There exists a maximal vector \( k \) for \( \ker T_g^{E_1, E_2} \) such that \( wk \in \ker T_h^{E_1, E_2} \);
    \item[(2)] \( w \in \mathcal{M}^+(\ker T_g^{E_1, E_2}, E_1) \).
\end{itemize}
\end{theorem}

These results underscore the structural parallels between Toeplitz kernels over the unit disk and the upper half-plane, extending the theory of multipliers and kernel generation to the broader context of non-standard Hardy subspaces.

\subsection{Multipliers between Toeplitz kernels in the upper-half plane setup}\label{sec4}
As model spaces can be realized as kernels of specific Toeplitz operators, a natural line of inquiry arises: can we characterize multipliers between general Toeplitz kernels in a way that simultaneously illuminates the conditions under which such kernels are non-trivial? One can observe that the following proposition is the particular case of Theorem (\ref{th8}).

\begin{Proposition}\label{thm1}
Let \( g, h \in L^\infty(\mathbb{R}) \) be non-zero functions such that the generalized Toeplitz kernels \( \ker T_g^{\mathcal{H}^2_+, E_2} \) and \( \ker T_h^{\mathcal{H}^2_+, E_2} \) are both non-trivial. Then
\[
hg^{-1} \in \overline{\mathcal{N}^+} \quad \Longleftrightarrow \quad \ker T_g^{\mathcal{H}^2_+, E_2} \subset \ker T_h^{\mathcal{H}^2_+, E_2}.
\]
\end{Proposition}

\begin{proof}
Suppose \( hg^{-1} \in \overline{\mathcal{N}^+} \), and let \( f \in \ker T_g^{\mathcal{H}^2_+, E_2} \). Then \( f = g^{-1}r \) for some \( r \in (E_2)^\perp \), and hence \( fh = r(hg^{-1}) \in (E_2)^\perp \), which implies \( f \in \ker T_h^{\mathcal{H}^2_+, E_2} \).

Conversely, assume \( \ker T_g^{\mathcal{H}^2_+, E_2} \subset \ker T_h^{\mathcal{H}^2_+, E_2} \). Let \( k \) be a maximal vector for \( \ker T_g^{\mathcal{H}^2_+, E_2} \). By Proposition ~\ref{thm4}, we have \( k = g^{-1}q\overline{p} \), where \( p \) is an outer function in \( \mathcal{H}^2_+ \). Since \( k \in \ker T_h^{\mathcal{H}^2_+, E_2} \), it follows that \( kh = hg^{-1}q\overline{p} \in (E_2)^\perp \), hence \( hg^{-1} \in \overline{\mathcal{N}^+} \).
\end{proof}

The idea behind the proof of the next theorem follows the framework outlined in \cite{FRR18}.

\begin{theorem}\label{thm3}
Let \( g, g^{-1}, h \in L^\infty(\mathbb{R}) \), and assume the Toeplitz kernels \( \ker T_g^{\mathcal{H}^2_+, E_2} \) and \( \ker T_h^{\mathcal{H}^2_+, E_2} \) are both non-trivial. Suppose further that \( \frac{h}{g}b_i \in \overline{\mathcal{N}^+} \), where \( b_i \) denotes the Blaschke factor vanishing at \( i \). Then the following statements are equivalent:
\begin{enumerate}
    \item[\textnormal{(1)}] \( \dim \ker T_{\frac{h}{g} \overline{b_i}}^{\mathcal{H}^2_+, E_2} \geq 2 \);
    \item[\textnormal{(2)}] \( \ker T_{\frac{h}{g}}^{\mathcal{H}^2_+, E_2} \neq \{0\} \);
    \item[\textnormal{(3)}] \( \mathcal{M}_2^+(\ker T_g^{\mathcal{H}^2_+, E_2}, \ker T_h^{\mathcal{H}^2_+, E_2}) \neq \{0\} \).
\end{enumerate}
\end{theorem}

\begin{proof}
(1) \( \Rightarrow \) (2): If \( \dim \ker T_{\frac{h}{g} \overline{b_i}}^{\mathcal{H}^2_+, E_2} \ge 2 \), then there exists \( f \in \ker T_{\frac{h}{g} \overline{b_i}}^{\mathcal{H}^2_+, E_2} \), non-zero, such that \( f(i) = 0 \). As \( f \in \mathcal{H}^2_+ \), it admits a factorization \( f = b_i \varphi \), with \( \varphi \in \mathcal{H}^2_+ \). Since
\[
\varphi \cdot \frac{h}{g} = \left( \overline{b_i} f \right) \cdot \frac{h}{g} \in (E_2)^\perp,
\]
we conclude \( \varphi \in \ker T_{\frac{h}{g}}^{\mathcal{H}^2_+, E_2} \), proving (2).

(2) \( \Rightarrow \) (3): By Proposition~\ref{thm1}, \( \ker T_{\overline{b_i}}^{\mathcal{H}^2_+, E_2} \subset \ker T_{\frac{h}{g}}^{\mathcal{H}^2_+, E_2} \). Since \( \ker T_{\overline{b_i}} \) contains bounded functions, there exists \( \varphi \in \ker T_{\frac{h}{g}}^{\mathcal{H}^2_+, E_2} \cap \mathcal{H}_+^\infty \) such that \( \varphi \overline{b_i} = \overline{r} \) for some \( r \in \mathcal{H}^2_+ \). Let \( f \in \ker T_g^{\mathcal{H}^2_+, E_2} \), and write \( f = q\overline{s} g^{-1} \) for some \( s \in \mathcal{H}^2_+ \). Then,
\[
\varphi f h \overline{q} = \left( \frac{\overline{r}}{\overline{b_i}} \cdot \frac{q \overline{s}}{g} \right) h \overline{q} \in \overline{\mathcal{N}^+},
\]
so \( \varphi f \in \ker T_h^{\mathcal{H}^2_+, E_2} \), and thus \( \varphi \in \mathcal{M}_2^+(\ker T_g^{\mathcal{H}^2_+, E_2}, \ker T_h^{\mathcal{H}^2_+, E_2}) \).

(3) \( \Rightarrow \) (1): Let \( \psi \in \mathcal{M}_2^+(\ker T_g^{\mathcal{H}^2_+, E_2}, \ker T_h^{\mathcal{H}^2_+, E_2}) \), and let \( k \) be a maximal vector for \( \ker T_g^{\mathcal{H}^2_+, E_2} \). By Proposition~\ref{thm4}, we have \( k = g^{-1} q \overline{p} \) with \( p \) outer. Since \( \psi k \in \ker T_h^{\mathcal{H}^2_+, E_2} \), it follows that
\[
\psi = \frac{g}{h} \cdot \frac{\overline{t}}{\overline{p}}, \quad \text{for some } t \in \mathcal{H}^2_+,
\]
and hence
\[
\psi \cdot \frac{h}{g} \cdot \overline{b_i q} \in \overline{\mathcal{N}^+}.
\]
Thus \( \psi \in \ker T_{\frac{h}{g} \overline{b_i}}^{\mathcal{H}^2_+, E_2} \). Moreover, since \( \ker T_{\overline{b_i^2}} \subset \ker T_{\overline{b_i^2}}^{\mathcal{H}^2_+, E_2} \subset \ker T_{\frac{h \overline{b_i}}{g}}^{\mathcal{H}^2_+, E_2} \), it follows that
\[
\dim \ker T_{\frac{h}{g} \overline{b_i}}^{\mathcal{H}^2_+, E_2} \ge 2.
\]
\end{proof}

We conclude this section by presenting explicit constructions that illustrate the applicability of Theorem~\ref{thm3} in diverse settings. We now present examples that illustrate the depth of this theory and its connections to spectral synthesis, zero set distribution, and Beurling-Malliavin theory. In particular, we expose how the existence of multipliers hinges on density conditions that appear naturally in classical spectral gap problems and the theory of Pólya sequences.

\begin{Example}
	\normalfont
Let \( g = b_i \overline{O} \), where \( O \) is the outer function associated with \( f(x) = e^{\frac{1}{1 + x^2}} \), and let \( h = \overline{S_{\delta_0}} \), with \( \delta_0 \) denoting the Dirac delta measure at the origin. Observe that
\[
\frac{h}{g} b_i = \frac{\overline{S_{\delta_0}}}{\overline{O}} \in \overline{\mathcal{N}^+},
\]
so the hypotheses of Theorem~\ref{thm3} are fulfilled. Consequently, we obtain the strict inclusion
\[
\mathcal{M}_2^+(\ker T_{b_i \overline{O}}, \ker T_{\overline{S_{\delta_0}}}) \neq \{0\},
\]
since the intermediate kernel \( \ker T_{\frac{\overline{S_{\delta_0}} b_i}{\overline{O}}} \) contains the non-trivial model space \( \ker T_{\overline{b_i}} \).
\end{Example}

\begin{Example}
	\normalfont
Let \( g = \overline{P u }\) and \( h = \overline{O b_i u} \), where \( O \) is as in the previous example, \( P \) is any outer function satisfying \( 1 \leq |P(x)| \leq A \) for some constant \( A > 0 \), and $u$ is any inner function. Then
\[
\frac{h b_i}{g} = \frac{\overline{O}}{\overline{P}} \in \overline{\mathcal{N}^+},
\]
so Theorem~\ref{thm3} ensures that
\[
\ker T_{\frac{\overline{O b_i}}{\overline P} }\neq \{0\} \quad \Longleftrightarrow \quad \mathcal{M}_2^+(\ker T_{\overline{P u}}, \ker T_{\overline{O b_i u}}) \neq \{0\}.
\]
This example highlights how bounded perturbations of outer functions preserve the existence of non-trivial multipliers between Toeplitz kernels.
\end{Example}

\begin{Example}\label{example4}
	\normalfont
Let \( g = S^{-b} b_i \) and \( h = S^{-a} \overline{\Theta} \), where \( b > a > 0 \), and \( \Theta \) is an inner function satisfying
\[
S^{b - a} \overline{\Theta} \in \overline{\mathcal{N}^+}.
\]
Assume further that \( \sigma(\Theta b_i) = \sigma(\Theta_1) \coloneqq \{ x \in \mathbb{R} : (\Theta b_i)(x) = 1 \} = \Lambda \), and let \( D = D_*(\Lambda) \) denote the Beurling Malliavin density of \( \Lambda \). Then the following dichotomy holds:
\[
\mathcal{M}_2^+(\ker T_g, \ker T_h) \neq \{0\} \quad \Longleftrightarrow \quad b - a < 2\pi D.
\]

Indeed, if \( b - a < 2\pi D \), then by the definition of \( D \), the kernel
\[
\ker T_{\frac{h}{g}} = \ker T_{S^{b - a} \overline{\Theta_1}}
\]
is non-trivial, and Theorem~\ref{thm3} implies the existence of non-zero multipliers. Conversely, if \( b - a > 2\pi D \), there exists \( \beta < b - a \) such that
\[
\ker T_{S^{\beta} \overline{\Theta_1}} = \{0\},
\]
which forces \( \ker T_{\frac{h}{g}} = \{0\} \), thereby yielding \( \mathcal{M}_2^+(\ker T_g, \ker T_h) = \{0\} \) via Theorem~\ref{thm3}.

This example underscores the intimate connection between multiplier theory and Beurling-Malliavin density thresholds. Some idea for this example has been taken from \cite{FRR18}.
\end{Example}

The following results are drawn from \cite{Mitkovski2010,Mitkovski2011,natalia2016} to highlight the connection between the multiplier space of two Toeplitz kernels and P\'olya sequences, as well as the Cartwright class for certain entire functions. We believe this area holds potential for further exploration.

\begin{theorem}\cite{Mitkovski2010} \label{thmA}
	Let $\Lambda = \{\lambda_n\}_{n=-\infty}^{\infty} \subset \mathbb{R}$ be a separated sequence of real numbers. The following are equivalent:
	
	\begin{itemize}
		\item[(i)] $\Lambda = \{\lambda_n\}$ is a Pólya sequence.
		\item[(ii)] There exists a non-zero measure $\mu$ of finite total variation, supported on $\Lambda$, such that the Fourier transform of $\mu$ vanishes on an interval of positive length.
		\item[(iii)] The interior Beurling--Malliavin density of $\Lambda$, $D_{*}(\Lambda)$, is positive.
		\item[(iv)] There exists a meromorphic inner function $\Theta(z)$ with $\{\Theta = 1\} = \{\lambda_n\}$ such that
		\[
		 \ker T_{\overline{\Theta} S^{2c}} \neq 0, \quad \text{for some } c > 0.
		\]
	\end{itemize}
\end{theorem}

Denote by $D^-(\Lambda)$ the lower uniform density of $\Lambda$:
\[D^-(\Lambda) \coloneqq \lim_{r \to \infty }\min_{x\in \mathbb R}\frac{\#(\Lambda\cap(x-r,x+r))}{2r}.\]

\begin{lemma}\label{lemma4}\cite{Mitkovski2011}
	Let $\Lambda = \{\lambda_n\}_{n \in \mathbb Z}$ be a sequence of real numbers. If $(e^{i\lambda_n t})$  is a Riesz
	basis in $L^2(0,c)$ then $c=D^-(\Lambda)=D_*(\Lambda)$.
\end{lemma}
\begin{theorem}\label{thmB}\cite{natalia2016}
	A set $\Lambda \subset \mathbb{R}$ is a Cartwright Set for entire functions of exponential type $<\sigma$ if and only if it contains a u.d.\ subset $\Lambda '$ satisfying $D^{-}(\Lambda ') \geq \sigma/\pi$.
\end{theorem}
Based on the above results, we make the following observation, which sheds light on the underlying relationship between Toeplitz kernel multipliers and classical concepts such as P\'olya sequences and Cartwright  sets. This insight not only reinforces the analytical framework built by earlier works but also suggests new directions for investigating the structure of multiplier spaces in relation to entire function theory.
\begin{remark}
\begin{itemize}
	
\smallskip	
	
\item[(a)] The space of multipliers can be related to the interior Beurling-Malliavin density. By applying Theorem \ref{thmA} \cite{Mitkovski2010}, we observe that the set \( \Lambda \), as introduced in Example~\ref{example4}, constitutes a P\'olya sequence. This implies the existence of a finite, non-zero measure \( \mu \) supported on \( \Lambda \) such that the Fourier transform of \( \mu \) vanishes on a non-trivial interval.

\item[(b)] Building on the results in Lemma \ref{lemma4} \cite{Mitkovski2011}, and Theorem \ref{thmB} \cite{natalia2016} consider \( g, h \in L^\infty(\mathbb{R}) \), as defined in Example~\ref{example4}, where \( \mathcal{M}_2^+(\ker T_g, \ker T_h) \neq \{0\} \) and \( \Lambda = \{\lambda_n\}_{n \in \mathbb{Z}} \) is a uniformly discrete set. If the sequence \( \{ e^{i \lambda_n t} \} \) forms a Riesz basis in \( L^2(0, D) \), then \( \Lambda \) is a Cartwright set for entire functions of exponential type less than or equal to \( \frac{b-a}{2} \).
\end{itemize}
\end{remark}
\section*{Acknowledgement}

We would like to thank Dr. A. K. Bhardwaj for his valuable suggestions and insightful remarks.

\bibliographystyle{plain}
	
\end{document}